\documentclass[12pt]{article}
\usepackage{graphicx}
\usepackage{amsmath,amsthm,amssymb,amsfonts,euscript,enumerate}
\usepackage{amsthm}
\usepackage{color,wrapfig}
\usepackage{pxfonts}
\usepackage{verbatim}
\newtheorem{thm}{Theorem}[section]
\newtheorem{cor}[thm]{Corollary}
\newtheorem{lem}[thm]{Lemma}

\newcommand{\R}{{\mathbb{R}}}
\newcommand{\Z}{{\mathbb{Z}}}

\newcommand{\2}{\overline}

\begin{document}
\title{Singular limit and exact decay rate of a nonlinear elliptic equation} 
\author{Shu-Yu Hsu\\
Department of Mathematics\\
National Chung Cheng University\\
168 University Road, Min-Hsiung\\
Chia-Yi 621, Taiwan, R.O.C.\\
e-mail: syhsu@math.ccu.edu.tw}
\date{July 14, 2011}
\smallbreak \maketitle
\begin{abstract}
For any $n\ge 3$, $0<m\le (n-2)/n$, and constants $\eta>0$, $\beta>0$,
$\alpha$, satisfying $\alpha\le\beta(n-2)/m$, we prove the existence of 
radially symmetric solution of $\frac{n-1}{m}\Delta v^m+\alpha v
+\beta x\cdot\nabla v=0$, $v>0$, in $\R^n$, $v(0)=\eta$, without using 
the phase plane method. When $0<m<(n-2)/n$, 
$n\ge 3$, and $\alpha=2\beta/(1-m)>0$, we prove that the radially 
symmetric solution $v$ of the above elliptic equation satisfies 
$\lim_{|x|\to\infty}\frac{|x|^2v(x)^{1-m}}{\log |x|}
=\frac{2(n-1)(n-2-nm)}{\beta(1-m)}$. In particular when $m=\frac{n-2}{n+2}$, 
$n\ge 3$, and $\alpha=2\beta/(1-m)>0$, the metric $g_{ij}=v^{\frac{4}{n+2}}dx^2$ 
is the steady soliton solution of the Yamabe flow on $\R^n$ and we obtain 
$\lim_{|x|\to\infty}\frac{|x|^2v(x)^{1-m}}{\log |x|}=\frac{(n-1)(n-2)}{\beta}$.
When $0<m<(n-2)/n$, $n\ge 3$, and $2\beta/(1-m)>\max (\alpha,0)$, we prove that
$\lim_{|x|\to\infty}|x|^{\alpha/\beta}v(x)=A$ for some constant $A>0$. 
For $\beta>0$ or $\alpha=0$, we prove that the radially symmetric solution 
$v^{(m)}$ of the above elliptic elliptic equation converges uniformly on every 
compact subset of $\R^n$ to the solution $u$ of the equation 
$(n-1)\Delta\log u+\alpha u+\beta x\cdot\nabla u=0$, $u>0$, in $\R^n$, 
$u(0)=\eta$, as $m\to 0$. 
\end{abstract}

\vskip 0.2truein


\vskip 0.2truein
\setcounter{equation}{0}
\setcounter{section}{-1}

\section{Introduction}
\setcounter{equation}{0}
\setcounter{thm}{0}

Recently there is a lot of interest in the following singular diffusion 
equation \cite{A}, \cite{DK}, \cite{P}, 
\begin{equation}\label{diffusion-eqn}
u_t=\frac{n-1}{m}\Delta u^m\quad\mbox{ in }\R^n\times (0,T)
\end{equation}
which arises in the study of many physical models.
When $m>1$, \eqref{diffusion-eqn} is called the porous medium equation which 
models the the flow of gases through porous medium. When $m=1$, 
\eqref{diffusion-eqn} is the well known heat equation with diffusivity 
coefficient equal to $(n-1)/m$. When $0<m<1$, \eqref{diffusion-eqn} is 
called the fast diffusion equation. Interested reader can read the book 
\cite{DK} by P.~Daskalopoulos and C.E.~Kenig and the book \cite{V1} by 
J.L.~Vazquez for the most recent results on \eqref{diffusion-eqn}. 

For any $n\in\Z^+$, $n\ge 3$, $0<m<1$, $\eta>0$, suppose $v$ is the solution 
of 
\begin{equation}\label{elliptic-eqn}
\left\{
\begin{aligned}
&\frac{n-1}{m}\Delta v^m
+\alpha v+\beta x\cdot\nabla v=0, v>0,\quad\mbox{ in }\R^n\\
&v(0)=\eta.\end{aligned}\right.
\end{equation}
Then as observed by B.H.~Gilding and L.A.~Peletier \cite{GP} and others
\cite{DS}, \cite{V1}, \cite{V2},
the function
\begin{equation*}
u_1(x,t)=t^{-\alpha}v(xt^{-\beta})
\end{equation*}
is a solution of \eqref{diffusion-eqn} in $\R^n\times (0,\infty)$ if 
\begin{equation}\label{forward}
\alpha=\frac{2\beta-1}{1-m}
\end{equation}
and for any $T>0$ the function 
\begin{equation*}
u_2(x,t)=(T-t)^{\alpha}v(x(T-t)^{\beta})
\end{equation*}
is a solution of \eqref{diffusion-eqn} in $\R^n\times (0,T)$ if 
\begin{equation}\label{backward}
\alpha=\frac{2\beta+1}{1-m}>0
\end{equation}
and the function 
\begin{equation*}
u_3(x,t)=e^{-\alpha t}v(xe^{-\beta t})
\end{equation*}
is an eternal solution of \eqref{diffusion-eqn} in 
$\R^n\times (-\infty,\infty)$ if 
\begin{equation}\label{eternal}
\alpha=\frac{2\beta}{1-m}.
\end{equation}
On the other hand P.~Daskalopoulos and N.~Sesum \cite{DS} proved that a
locally conformally flat gradient Yamabe soliton with positive sectional 
curvature must be radially symmetric and 
the metric $g_{ij}=v^{\frac{4}{n+2}}dx^2$ satisfies \eqref{elliptic-eqn} or
\begin{equation}\label{ode}
\frac{n-1}{m}\left((v^m)''+\frac{n-1}{r}(v^m)'\right)
+\alpha v+\beta rv'=0, v>0,
\end{equation}
in $(0,\infty)$ and 
\begin{equation}\label{initial-cond}
\left\{\begin{aligned}
&v(0)=\eta\\
&v'(0)=0\end{aligned}\right.
\end{equation}
for some constant $\eta>0$ where $dx^2$ is the standard metric on $\R^n$ 
with $m=(n-2)/(n+2)$, $n\ge 3$, and
\begin{equation}\label{alpha-beta-1}
\alpha=\frac{2\beta+\rho_1}{1-m}
\end{equation}
for some constants $\beta>0$, $\alpha$, and $\rho_1$ where $\rho_1=0$ if 
$g_{ij}$ is a Yamabe steady soliton, $\rho_1<0$ if $g_{ij}$ is a Yamabe 
expander soliton, and $\rho_1>0$ if $g_{ij}$ is a Yamabe shrinker soliton.

Since the asymptotic behaviour of the solutions of \eqref{diffusion-eqn}
are usually similar to either the functions $u_1$, $u_2$ or $u_3$, it is
important to study the solutions of \eqref{elliptic-eqn} in order to 
understand the behaviour of solutions of \eqref{diffusion-eqn} and the
locally conformally flat gradient Yamabe solitons. Existence and uniqueness 
of radially symmetric solution of \eqref{elliptic-eqn} for $\alpha$, $\beta$, 
satisfying \eqref{backward} and \begin{equation}\label{m-range}
0<m<\frac{n-2}{n}, n\ge 3,
\end{equation}
is proved by M.A.~Peletier, H.~Zhang \cite{PZ} and J.R.~King \cite{K} using 
phase plane method (cf. Proposition 7.4 of \cite{V1}). Existence of radially
symmetric solution of \eqref{elliptic-eqn} for $\alpha$, $\beta>0$, satisfying 
\eqref{forward} and \eqref{m-range} is proved on P.22 of \cite{DS}. A sketch 
of the proof of the existence of radially symmetric solution of 
\eqref{elliptic-eqn} for $m=(n-2)/(n+2)$, $n\ge 3$, and $\alpha$, $\beta>0$, 
satisfying \eqref{eternal} is given on P.22-23 of \cite{DS}. This existence
result is also noted without proof in \cite{GaP}. 

In \cite{DS} P.~Daskalopoulos and N.~Sesum also 
proved that if $m=(n-2)/(n+2)$, $n\ge 6$ and $\alpha$, $\beta>0$, satisfy 
\eqref{eternal}, then the radially symmetric solution of \eqref{elliptic-eqn} 
satisfies  
\begin{equation}\label{v-ineqn}
C_1\frac{\log |x|}{|x|^2}\le v(x)^{1-m}\le C_2\frac{\log |x|}{|x|^2}
\qquad\qquad\mbox{ as }|x|\to\infty
\end{equation}
for some constants $C_2>C_1>0$.

In this paper we will extend the result of \cite{DS} and give a new simple 
rigorous proof of the existence of radially symmetric solutions of 
\eqref{elliptic-eqn} for any $\eta>0$ and $\alpha$, $\beta$, $n$, $m$, 
satisfying 
\begin{equation}\label{m-range2}
0<m\le\frac{n-2}{n},\quad n\ge 3,
\end{equation}
and 
\begin{equation}\label{alpha-beta-relation}
\alpha\le\frac{\beta (n-2)}{m}\quad\mbox{ and }\quad\beta>0
\end{equation}
without using the phase plane method. Note that if \eqref{m-range2} holds, 
then \eqref{alpha-beta-relation} holds if $\beta>0$ and 
$$
\alpha\le\frac{2\beta}{1-m}
$$
hold. For 
\begin{equation}\label{alpha-beta-relation2}
\beta>0\quad\mbox{ or }\quad\alpha=0,
\end{equation}
we prove that the radially symmetric 
solution $v^{(m)}$ of \eqref{elliptic-eqn} converges uniformly on every 
compact subset of $\R^n$ to the solution $u$ of the equation 
\begin{equation}\label{log-eqn}
\left\{\begin{aligned}
&(n-1)\Delta\log u+\alpha u+\beta x\cdot\nabla u=0, u>0,\mbox{ in }\R^n\\
&u(0)=\eta
\end{aligned}\right.\\
\end{equation}
as $m\to 0$. When $\alpha$, $\beta$, $m$, satisfy \eqref{m-range} and 
\begin{equation}\label{alpha-beta-relation3}
\alpha=2\beta/(1-m)>0,
\end{equation} 
we prove that the radially symmetric solution $v$ of \eqref{elliptic-eqn} 
satisfies
\begin{equation}\label{decay-rate2}
\lim_{|x|\to\infty}\frac{|x|^2v(x)^{1-m}}{\log |x|}
=\frac{2(n-1)(n-2-nm)}{\beta(1-m)}.
\end{equation}
When $m=(n-2)/(n+2)$ and \eqref{alpha-beta-relation3} hold, this 
result says that the locally conformally flat gradient steady Yamabe solitons 
$g_{ij}=v^{\frac{4}{n+2}}dx^2$, $n\ge 3$, has exact decay rate 
\begin{equation}\label{decay-rate}
\lim_{|x|\to\infty}\frac{|x|^2v(x)^{1-m}}{\log |x|}=\frac{(n-1)(n-2)}{\beta}.
\end{equation}
In Theorem 3.2 of \cite{V1} J.L.Vazquez by using phase plane method 
proved that if \eqref{forward} and \eqref{m-range} holds, then the 
radially symmetric solution $v$ of \eqref{elliptic-eqn} satisfies 
\begin{equation}\label{decay-rate3}
\lim_{|x|\to\infty}|x|^{\alpha/\beta}v(x)=A
\end{equation}
for some constant $A>0$. In this paper we will extend this theorem and use a 
modification of the technique of \cite{Hs} to give a new simple proof of
the result that if \eqref{m-range} and 
\begin{equation}\label{alpha-beta-relation5}
\frac{2\beta}{1-m}>\max (\alpha,0)
\end{equation}
hold and $v$ is the radially symmetric solution of \eqref{elliptic-eqn}, 
then \eqref{decay-rate3} for some constant $A>0$. 

The plan of the paper is as follows. In section 1 we will prove the 
existence of radially symmetric solutions of \eqref{elliptic-eqn} when
\eqref{m-range2} and \eqref{alpha-beta-relation} hold. We will also prove
the singular limit of the radially symmetric solution of 
\eqref{elliptic-eqn} as $m\to 0$. In section 2 we will prove the exact decay 
rate \eqref{decay-rate2} of the radially symmetric solution of 
\eqref{elliptic-eqn} when \eqref{m-range} and \eqref{alpha-beta-relation3} 
hold. In section 3 we will prove the decay rate \eqref{decay-rate3} of the 
radially symmetric solution of \eqref{elliptic-eqn} when \eqref{m-range} and
\eqref{alpha-beta-relation5} hold. 
We let
$$
k=\frac{\beta}{\alpha}\quad\mbox{ if }\alpha\ne 0.
$$
and we will assume that \eqref{m-range2} holds for the rest of the paper. 

\section{Existence and singular limit of solutions}
\setcounter{equation}{0}
\setcounter{thm}{0}

In this section we will prove the existence of radially symmetric solutions 
of \eqref{elliptic-eqn} and the singular limit of radially 
symmetric solutions of \eqref{elliptic-eqn} as $m\to 0$.

\begin{lem}\label{h1-lower-bd}
Let $m$, $\alpha\ne 0$, $\beta\ne 0$, satisfy \eqref{m-range2} and 
\begin{equation}\label{alpha-beta-relation4}
\frac{m\alpha}{\beta}\le n-2.
\end{equation}
For any $R_0>0$ and $\eta>0$, let $v$ be the solution of \eqref{ode}, 
\eqref{initial-cond}, in $(0,R_0)$. Then  
\begin{equation}\label{basic-monotone-ineqn}
v+krv'(r)>0\quad\mbox{ in }[0,R_0)
\end{equation}
and
\begin{equation}\label{v'-bd0}
\left\{\begin{aligned}
&v'(r)<0\quad\mbox{ in }(0,R_0)\quad\mbox{ if }\alpha>0\\
&v'(r)>0\quad\mbox{ in }(0,R_0)\quad\mbox{ if }\alpha<0.\end{aligned}\right.
\end{equation}
\end{lem} 
\begin{proof}
Let $h_1(r)=v(r)+krv'(r)$. By \eqref{alpha-beta-relation4}, $(n-2)\ge m/k$. 
Then by direct computation,  
\begin{equation}\label{h1-eqn}
h_1'+\left(\frac{(n-2)-(m/k)}{r}-(1-m)\frac{v'}{v}+\frac{\beta}{n-1}rv^{1-m}
\right)h_1=\frac{(n-2)-(m/k)}{r}v\ge 0\quad\mbox{ in }(0,R_0).
\end{equation}
Let 
\begin{equation}\label{f-defn}
f(r)=v(r)^{m-1}exp\left(\frac{\beta}{n-1}\int_0^r\rho v(\rho)^{1-m}\,d\rho\right).
\end{equation}
By \eqref{h1-eqn}, 
\begin{align*}
&(r^{n-2-(m/k)}f(r)h_1(r))'\ge 0\quad\forall 0<r<R_0\nonumber\\
\Rightarrow\quad&r^{n-2-(m/k)}f(r)h_1(r)>0\qquad\forall 0<r<R_0\nonumber\\
\Rightarrow\quad&h_1(r)>0\qquad\qquad\qquad\quad\,\forall 0\le r<R_0
\end{align*} 
and \eqref{basic-monotone-ineqn} follows. By \eqref{ode},
\eqref{initial-cond}, and \eqref{basic-monotone-ineqn},
\begin{align*}
&\frac{n-1}{m}\frac{1}{r^{n-1}}(r^{n-1}(v^m)')'=-\alpha h_1
\left\{\begin{aligned}
&<0\quad\mbox{ in }(0,R_0)\quad\mbox{ if }\alpha>0\\
&>0\quad\mbox{ in }(0,R_0)\quad\mbox{ if }\alpha<0
\end{aligned}\right.\\
\Rightarrow\quad&\left\{\begin{aligned}
&r^{n-1}(v^m)'<0\quad
\mbox{ in }(0,R_0)\quad\mbox{ if }\alpha>0\\
&r^{n-1}(v^m)'>0\quad\mbox{ in }(0,R_0)\quad\mbox{ if }\alpha<0
\end{aligned}\right.
\end{align*}
and \eqref{v'-bd0} follows.
\end{proof}

\begin{thm}\label{existence-thm}
Let $\eta>0$ and let $\alpha,\beta\in\R$, $m$, satisfy \eqref{m-range2}
and \eqref{alpha-beta-relation}. Then there exists a unique solution $v$ of 
\eqref{ode}, \eqref{initial-cond}, in $(0,\infty)$. Moreover the function
\begin{equation}\label{w-defn}
w_1(r)=r^{2}v(r)^{2k}
\end{equation}
satisfies $w_1'(r)>0$ for all $r>0$.
\end{thm}
\begin{proof}
We will use a modification of the proof of Theorem 1.3 of \cite{Hs} to prove 
the theorem. If $\alpha=0$, the constant function $v(r)\equiv\eta$ is 
the unique solution of \eqref{ode}, \eqref{initial-cond}, in $(0,\infty)$
and then $w_1(r)=\eta^{2k}r^2$ satisfies $w'(r)>0$ for any $r>0$. 
Hence we may assume $\alpha\ne 0$ in the proof. 

We next note that uniqueness 
of solution of \eqref{ode}, \eqref{initial-cond}, in $(0,\infty)$ follows 
by standard O.D.E. theory. Hence we only need to prove existence of solution 
of \eqref{ode}, \eqref{initial-cond}, in $(0,\infty)$. Local existence of 
solution of \eqref{ode}, \eqref{initial-cond}, in a neighbourhood of the 
origin follows by standard O.D.E. theory.
   
Let $(0,R_0)$ be the maximal interval of existence of solution of \eqref{ode}, 
\eqref{initial-cond}. Suppose $R_0<\infty$. Then there exists a sequence 
$\{r_i\}_{i=1}^{\infty}$, $r_i\nearrow R_0$ as $i\to\infty$, such that 
either 
$$
|v'(r_i)|\to\infty\,\,\mbox{as}\,\, i\to\infty\quad\mbox{ or }\quad v(r_i)
\searrow 0\,\,\mbox{as}\,\, i\to\infty\quad\mbox{ or }\quad
v(r_i)\to\infty\,\, \mbox{as}\,\, i\to\infty.
$$
By \eqref{alpha-beta-relation}, \eqref{alpha-beta-relation4} holds. Hence by 
Lemma \ref{h1-lower-bd},
\begin{equation}\label{w'-eqn}
w_1'(r)=2rv^{2k}+2kr^2v^{2k-1}v'=2rv^{2k-1}(v+krv')>0
\quad\forall 0<r<R_0.
\end{equation}
We now divide the proof into two cases.

\noindent $\underline{\text{\bf Case 1}}$: $\alpha>0$.

\noindent By \eqref{w'-eqn},
\begin{align}\label{v-lower-bd}
&w_1(r)=r^2v^{2k}\ge w_1(R_0/2)>0\quad\forall R_0/2\le r<R_0\nonumber\\
\Rightarrow\quad&v(r)\ge (R_0^{-2}w_1(R_0/2))^{\frac{1}{2k}}\quad\forall R_0/2
\le r<R_0. 
\end{align} 
By Lemma \ref{h1-lower-bd} $v'<0$ on $(0,R_0)$. Hence 
\begin{equation}\label{v-upper-bd}
0<v(r)\le v(0)=\eta\quad\forall 0\le r<R_0.
\end{equation}
By \eqref{ode}, \eqref{initial-cond}, and \eqref{v-upper-bd},
\begin{align}
&\frac{n-1}{m}\frac{1}{r^{n-1}}(r^{n-1}(v^m)')'=-(\alpha v+\beta rv')\quad
\mbox{ in }(0,R_0)\nonumber\\
\Rightarrow\quad&(n-1)r^{n-1}(v^m)'=-m\left(\alpha\int_0^r\rho^{n-1}v(\rho)
\,d\rho
+\beta\int_0^r\rho^nv'(\rho)\,d\rho\right)\quad\mbox{ in }(0,R_0)\nonumber\\
\Rightarrow\quad&(n-1)(v^m/m)'=-\beta rv(r)+\frac{(n\beta-\alpha)}{r^{n-1}}
\int_0^r\rho^{n-1}v(\rho)\,d\rho\qquad\qquad\,\,
\mbox{ in }(0,R_0)\label{vm'-bd}\\
\Rightarrow\quad&(n-1)v(r)^{m-1}|v'(r)|\le\left(\beta 
+\frac{|n\beta-\alpha|}{n}\right)R_0v(0)\qquad\qquad\qquad\qquad
\mbox{ in }(0,R_0)\nonumber\\
\Rightarrow\quad&(n-1)|v'(r)|\le\left(\beta 
+\frac{|n\beta-\alpha|}{n}\right)R_0v(0)^{2-m}\qquad\qquad\qquad\qquad\quad
\,\,\,\mbox{ in }(0,R_0)\label{v'-bd}.
\end{align}
By \eqref{v-lower-bd}, \eqref{v-upper-bd}, \eqref{v'-bd}, a contradiction 
arises. Hence no such sequence $\{r_i\}_{i=1}^{\infty}$ exists. Thus 
$R_0=\infty$ and there exists a unique solution of \eqref{ode}, 
\eqref{initial-cond},
in $(0,\infty)$.

\noindent $\underline{\text{\bf Case 2}}$: $\alpha<0$.

\noindent By Lemma \ref{h1-lower-bd},
\begin{equation}\label{v'-v-bd}
0<v'(r)\le\frac{v}{|k|r}\quad\mbox{ in }(0,R_0).
\end{equation}
By \eqref{v'-v-bd} and an argument similar to the proof of case 2 of 
Theorem 1.3 of \cite{Hs}, there exists a constant $C>0$ such that
\begin{equation*}
0<v'(r)\le Cv(r)\quad\forall 0\le r<R_0.
\end{equation*}
Then
\begin{equation}\label{v0-v-bd}
v(0)\le v(r)\le v(0)\mbox{exp}\,(CR_0)\quad\forall 0\le r<R_0
\end{equation}
and 
\begin{equation}\label{v'-v0-bd}
0<v'(r)\le Cv(0)\mbox{exp}\,(CR_0)\quad\forall 0\le r<R_0.
\end{equation}
By \eqref{v0-v-bd} and \eqref{v'-v0-bd} , a contradiction 
arises. Hence no such sequence $\{r_i\}_{i=1}^{\infty}$ exists. Thus 
$R_0=\infty$ and there exists a unique solution of \eqref{ode}, 
\eqref{initial-cond}, in $(0,\infty)$.
By case 1, case 2, and \eqref{w'-eqn} the lemma follows.
\end{proof}

\begin{thm}\label{m-to-0-thm}
Let $\eta>0$ and $m$, $n$, $\alpha$, $\beta$, satisfy \eqref{m-range2}
and \eqref{alpha-beta-relation2} and let $v^{(m)}$ be the radially symmetric 
solution of \eqref{elliptic-eqn}. Then $v^{(m)}$ converges uniformly on 
every compact subset of $\R^n$ to the solution of \eqref{log-eqn} as 
$m\to 0$.
\end{thm}
\begin{proof}
If $\alpha=0$, then $v^{(m)}\equiv\eta$ on $\R^n$ which satisfies 
\eqref{log-eqn} and we are done. Hence we may assume that $\alpha\ne 0$.
Then by \eqref{m-range2} and \eqref{alpha-beta-relation2} there exists 
a constant $m_0'\in (0,(n-2)/n)$ such that \eqref{alpha-beta-relation} holds 
for any $0<m\le m_0'$. Without loss of generality we may assume that 
$0<m\le m_0'$ in the proof. Note that $v^{(m)}(x)=v^{(m)}(|x|)$ satisfies 
\eqref{ode} and \eqref{initial-cond} in $(0,\infty)$. Let 
$\{m_i\}_{i=1}^{\infty}$ be a sequence such that $0<m_i<m_0'$ 
for all $i\in\Z^+$ and $m_i\to 0$ as $i\to\infty$. We now divide the proof 
into two cases.

\noindent $\underline{\text{\bf Case 1}}$: $\alpha>0$.

By the proof of Theorem \ref{existence-thm},
$v^{(m)}$ satisfies \eqref{v-upper-bd}, \eqref{vm'-bd}, and \eqref{v'-bd} in
$(0,\infty)$. Hence
\begin{equation}\label{v^{(m)}-bd}
0<v^{(m)}(r)\le\eta\quad\forall r\ge 0, 
\end{equation} 
\begin{align}\label{v^{(m)m}/m'-bd}
&(n-1)(v^{(m)m}/m)'=-\beta rv^{(m)}(r)+\frac{(\beta n-\alpha)}{r^{n-1}}
\int_0^r\rho^{n-1}v^{(m)}(\rho)\,d\rho\quad\mbox{ in }(0,\infty)\notag\\
\Rightarrow\quad&\frac{v^{(m)}(r)^m-1}{m}-\frac{\eta^m-1}{m}\notag\\
=&-\frac{\beta}{n-1}\int_0^r\rho v^{(m)}(\rho)\,d\rho +\frac{n\beta-\alpha}{n-1}
\int_0^r\frac{1}{\sigma^{n-1}}\left(\int_0^{\sigma}\rho^{n-1}v^{(m)}(\rho)
\,d\rho\right)\,d\sigma\quad\mbox{ in }(0,\infty),
\end{align} 
and for any $r_0>0$,
\begin{align}\label{v^{(m)}'-bd}
&(n-1)\left|\frac{d}{d r}v^{(m)}(r)\right|\le\left(\beta
+\frac{|n\beta-\alpha|}{n}\right)\eta^{2-m}r_0\quad\forall 0\le r\le r_0\notag\\
\Rightarrow\quad&|v^{(m)}(r_1)-v^{(m)}(r_2)|\le C_1r_0|r_1-r_2|
\quad\forall 0\le r_1,r_2\le r_0
\end{align} 
where
$$
C_1=(n-1)^{-1}(2\beta+(|\alpha|/n))\max(1,\eta)^2.
$$
By \eqref{v^{(m)}-bd} and \eqref{v^{(m)}'-bd}, the 
sequence $\{v^{(m_i)}\}_{i=1}^{\infty}$ is equi-Holder continuous on every 
compact subset of $[0,\infty)$. By the Ascoli Theorem the sequence 
$\{v^{(m_i)}\}_{i=1}^{\infty}$ has a subsequence which we may assume without loss
of generality to be the sequence itself that converges uniformly on every
compact subset of $[0,\infty)$ to some continuous function $u$  
as $i\to\infty$ and $u(0)=\eta$. By \eqref{v^{(m)}'-bd},
\begin{align*}
&|u(r)-u(0)|\le C_1r_0r\quad\forall 0<r\le r_0\\
\Rightarrow\quad&\limsup_{r\to 0}\left|\frac{u(r)-u(0)}{r}\right|\le Cr_0
\quad\forall r_0>0\\
\Rightarrow\quad&\lim_{r\to 0}\left|\frac{u(r)-u(0)}{r}\right|=0\quad\mbox{ as }
r_0\to 0.
\end{align*}
Hence $u$ is differentiable at $r=0$ with $u'(0)=0$. Thus
\begin{equation}\label{u-initial-cond}
u(0)=\eta,\quad u'(0)=0
\end{equation}
hold. By \eqref{v^{(m)}'-bd},
\begin{align}\label{u-lower-bd}
&v^{(m)}(r)\ge v^{(m)}(0)-(\eta/2)=\eta/2\quad\forall 0\le r\le
\min (1,\eta/(2C_1))\notag\\
\Rightarrow\quad&u(r)\ge\eta/2\quad\forall 0\le r\le
\min (1,\eta/(2C_1))\quad\mbox{ as }m=m_i\to\infty.
\end{align}
By \eqref{u-lower-bd} there exists a maximal interval $(0,R_1)$ such that 
$u(r)>0$ in $(0,R_1)$. Suppose $R_1<\infty$. Then $u(R_1)=0$. For any 
$0<\delta<R_1$, since
\begin{equation*}
\inf_{0\le r\le R_1-\delta}u(r):=c_0>0,
\end{equation*}
there exists $i_0\in\Z^+$ such that 
\begin{equation}\label{v^{(m)}-lower-bd}
v^{(m_i)}(r)\ge c_0/2\quad\forall 0\le r\le R_1-\delta, i\ge i_0.
\end{equation}
By \eqref{v^{(m)}-bd}, \eqref{v^{(m)}-lower-bd}, and the mean value theorem,
\begin{align}\label{v^{(m)}/m-limit}
\left|\frac{v^{(m_i)}(r)^{m_i}-1}{m_i}-\log u(r)\right|
=&|e^{\xi_i}\log v^{(m_i)}-\log u(r)|\notag\\
\le&e^{\xi_i}|\log v^{(m_i)}-\log u(r)|+|e^{\xi_i}-1||\log u(r)|\notag\\
\le&e^{m_iM}|\log v^{(m_i)}-\log u(r)|+|e^{\xi_i}-1||\log u(r)|
\notag\\
\to&0\quad\mbox{ uniformly on }[0,R_1-\delta]
\quad\mbox{ as }i\to\infty
\end{align}
for some $\xi_i$ satisfying $|\xi_i|\le m_iM$ for any $i\in\Z^+$
where $M=\max(|\log\eta|,|\log (c_0/2)|)$.
Putting $m=m_i$ in \eqref{v^{(m)m}/m'-bd} and letting $i\to\infty$, by 
\eqref{v^{(m)}/m-limit}, 
\begin{equation}\label{log-integral-eqn}
(n-1)\log u(r)=-\beta\int_0^r\rho u(\rho)\,d\rho +(n\beta-\alpha)
\int_0^r\frac{1}{\sigma^{n-1}}\left(\int_0^r{\sigma}^{n-1}u(\rho)\,d\rho\right)
\,d\sigma\quad\mbox{ in }(0,R_1).
\end{equation}
Since the right hand side of \eqref{log-integral-eqn} is a differentiable 
function of $r\in [0,R_1)$, $u(r)$ is a differentiable 
function of $r\in [0,R_1)$. Differentiating \eqref{log-integral-eqn} with
respect to $r$,
\begin{align}
&(n-1)\frac{u'(r)}{u(r)}=-\beta ru(r)+\frac{n\beta-\alpha}{r^{n-1}}
\int_0^r\rho^{n-1}u(\rho)\,d\rho\quad\mbox{ in }(0,R_1)
\label{log-integral-eqn2}\\
\Rightarrow\quad&(n-1)\left(\frac{r^{n-1}u'}{u}\right)
=-\alpha\int_0^r\rho^{n-1}u(\rho)\,d\rho-\beta\int_0^r\rho^nu'(\rho)\,d\rho
\quad\mbox{ in }(0,R_1)\label{log-integral-eqn3}.
\end{align} 
Since the right hand side of \eqref{log-integral-eqn3} is a differentiable 
function of $r\in [0,R_1)$, $u'(r)/u(r)$ is a differentiable 
function of $r\in [0,R_1)$. Differentiating \eqref{log-integral-eqn3} with
respect to $r$,
\begin{align}\label{log-eqn2}
&(n-1)\left(\frac{u'}{u}\right)'+\frac{n-1}{r}\frac{u'}{u}
=(n-1)\frac{1}{r^{n-1}}\left(\frac{r^{n-1}u'}{u}\right)'
=-\alpha u-\beta ru'\quad\mbox{ in }(0,R_1).
\end{align} 
Hence $u$ is a classical solution of \eqref{log-eqn2} and satisfies 
\eqref{u-initial-cond}.
By Theorem 1.3 of \cite{Hs} and a rescaling there exists a unique positive 
solution $\2{u}$ of \eqref{log-eqn2} in $[0,\infty)$ that satisfies 
\eqref{u-initial-cond}. By uniqueness of solution,
\begin{equation*}
u(r)\equiv\2{u}(r)\quad\forall 0\le r\le R_1\quad\Rightarrow\quad
u(R_1)\equiv\2{u}(R_1)>0
\end{equation*}
and contradiction arises. Hence $R_1=\infty$ and $u(r)>0$ for all $r\ge 0$.
By the above argument the solution $u$ satisfies \eqref{log-eqn2} and 
\eqref{u-initial-cond} and $u(r)\equiv\2{u}(r)$ is the a unique positive 
solution $\2{u}$ of \eqref{log-eqn2} in $[0,\infty)$ that satisfies 
\eqref{u-initial-cond}. Since the sequence $\{m_i\}_{i=1}^{\infty}$ is 
arbitrary, $v^{(m)}$ converges uniformly on every compact subset of $\R^n$
to the solution $u$ of \eqref{log-eqn} as $m\to 0$.

\noindent $\underline{\text{\bf Case 2}}$: $\alpha<0$.

By the proof of Theorem \ref{existence-thm},
$v^{(m)}$ satisfies \eqref{vm'-bd} and \eqref{v'-v-bd}  in $(0,\infty)$.
We choose $m_0\in (0,m_0']$ such that 
\begin{equation}\label{m-bd}
\frac{1}{2}\le (2\eta)^m\le 2\quad\forall 0<m\le m_0
\end{equation}
and let $r_0=\min((8C_2\eta)^{-\frac{1}{2}},(8C_2\eta)^{-1})$ where 
$C_2=(2\beta+(|\alpha|/n))/(n-1)$.
Let 
$$
r_m=\sup\{\delta'>0:v(r)\le 2\eta\quad\forall 0\le r\le\delta'\}.
$$
By \eqref{initial-cond}, $r_m>0$. We claim that 
\begin{equation}\label{rm-lower-bd}
r_m\ge r_0\quad\forall 0<m\le m_0.
\end{equation} 
Suppose \eqref{rm-lower-bd} does not hold. Then there exists $m'\in (0,m_0]$ 
such that $r_{m'}<r_0$. Then by \eqref{vm'-bd} and \eqref{m-bd},
\begin{align}
\left|\frac{dv^{(m')}}{dr}(r)\right|
\le&\frac{1}{n-1}\left(\beta rv(r)+\frac{n\beta-\alpha}{r^{n-1}}
\int_0^r\rho^{n-1}v(\rho)\,d\rho\right)v(r)^{1-m'}\quad\forall 0\le r\le r_{m'}
\notag\\
\Rightarrow\quad\left|\frac{dv^{(m')}}{dr}(r)\right|
\le&(n-1)^{-1}(2\beta +(|\alpha|/n))(2\eta)^{2-m'}r=8C_2\eta^2r
\quad\forall 0\le r\le r_{m'}\label{v'-bd20}\\
\Rightarrow\qquad v^{(m')}(r)\le&\eta+4C_2\eta^2r_0^2\le 3\eta/2
\quad\forall 0\le r\le r_{m'}.\label{v^{(m)}-bd2}
\end{align}
By \eqref{v^{(m)}-bd2} and continuity there exists a constant $\delta_1>0$ 
such that $v^{(m')}(r)\le 2\eta$ in $[0,r_{m'}+\delta_1]$. This contradicts the 
choice of $r_{m'}$. Hence no such $m'$ exists and \eqref{rm-lower-bd} holds.
By \eqref{v'-v-bd} and \eqref{v'-bd20},
\begin{align}
&0\le\frac{dv^{(m)}}{dr}(r)\le 8C_2\eta^2r_0=\eta\quad\forall 0\le r\le r_0,
0<m\le m_0\label{v^{(m)}'-bd3}\\
\Rightarrow\quad&\eta\le v^{(m)}(r)\le 2\eta\quad\forall 0\le r\le r_0,
0<m\le m_0\label{v^{(m)}-bd3}.
\end{align}
By \eqref{v'-v-bd} and \eqref{v^{(m)}-bd3},
\begin{align}\label{v^{(m)}-bd4}
&v^{(m)}(r_0)\le v^{(m)}(r)\le v^{(m)}(r_0)(r/r_0)^{\frac{1}{|k|}}\quad\forall
r\ge r_0,0<m\le m_0\notag\\
\Rightarrow\quad&\eta\le v^{(m)}(r)\le 2\eta(r/r_0)^{\frac{1}{|k|}}\quad\forall
r\ge r_0,0<m\le m_0.
\end{align}
By \eqref{v'-v-bd}, \eqref{v^{(m)}'-bd3}, \eqref{v^{(m)}-bd3} and 
\eqref{v^{(m)}-bd4}, for any $r_1>0$ there exists a constant $M_{r_1}>0$ 
such that
\begin{equation}\label{v-bd5}
\left\{\begin{aligned}
&0\le\frac{dv^{(m)}}{dr}(r)\le M_{r_1}\quad\forall 0\le r\le r_1,0<m\le m_0\\
&\eta\le v^{(m)}(r)\le M_{r_1}\quad\forall 0\le r\le r_1,0<m\le m_0.
\end{aligned}\right.
\end{equation}
By \eqref{v-bd5} the sequence 
$\{m_i\}_{i=1}^{\infty}$ is equi-Holder continuous on every compact subset of
$[0,\infty)$. By the Ascoli theorem the sequence $\{m_i\}_{i=1}^{\infty}$  
has a subsequence which we may assume without loss of generality to be the
sequence itself that converges uniformly to some continuous function
$u$ on every compact subset of $[0,\infty)$ as $i\to\infty$. By an argument 
similar to the proof of case 1 $u$ is the a unique positive 
solution $\2{u}$ of \eqref{log-eqn2} in $[0,\infty)$ that satisfies 
\eqref{u-initial-cond}. Since the sequence $\{m_i\}_{i=1}^{\infty}$ is 
arbitrary, $v^{(m)}$ converges uniformly on every compact subset of $\R^n$
to the solution $u$ of \eqref{log-eqn} as $m\to 0$ and the theorem follows.
\end{proof}

\section{Exact decay rate for $\alpha=\frac{2\beta}{1-m}>0$}
\setcounter{equation}{0}
\setcounter{thm}{0}

In this section we will prove the exact decay rate \eqref{decay-rate2} for the 
radially symmetric solution $v$ of \eqref{elliptic-eqn} when \eqref{m-range} 
and \eqref{alpha-beta-relation3} hold. We let 
\begin{equation*}
h(r)=v+\frac{1-m}{2}rv'(r)\quad\mbox{ and }\quad w(r)=r^2v(r)^{1-m}.
\end{equation*}

\begin{lem}\label{h-lower-bd}
Let $\eta>0$, and $\alpha$, $\beta$, $m$, satisfies \eqref{m-range2} and
\begin{equation}\label{alpha-beta-5}
\frac{2\beta}{1-m}\ge\alpha>0.
\end{equation} 
Let $v$ be the radially symmetric solution of \eqref{elliptic-eqn}. Then 
$h(r)>0$ for any $r\ge 0$ and $w'(r)>0$ for any $r>0$. 
\end{lem}  
\begin{proof}
By direct computation,
\begin{align}\label{h-eqn}
&h'(r)+\left(\frac{n-2-mn}{(1-m)r}-(1-m)\frac{v'}{v}+\frac{\beta}{(n-1)}rv^{1-m}
\right)h\notag\\
=&\frac{n-2-mn}{1-m}\cdot\frac{v}{r}
+\frac{1}{n-1}\left(\frac{2\beta}{1-m}-\alpha\right)rv^{2-m}\ge 0\quad
\forall r>0.
\end{align}
Let $f$ be given by \eqref{f-defn}. Then by \eqref{h-eqn},
$$
(r^{\frac{n-2-mn}{1-m}}f(r)h(r))'\ge 0\quad\Rightarrow\quad h(r)>0
\quad\forall r\ge 0.
$$
Hence
$$
w'(r)=2rv(r)^{-m}h(r)>0\quad\forall r>0
$$
and the lemma follows.
\end{proof}

Let $\eta>0$, and let $m$, $\alpha$, $\beta$, $\rho_1$, satisfy 
\eqref{alpha-beta-1} and \eqref{m-range}. Suppose $v$ is a radially 
symmetric solution of \eqref{elliptic-eqn}. Let $s=\log r$ and $v_1
=w^{\frac{1}{1-m}}$. Then $v_1$ satisfies
\begin{equation}\label{v1-eqn}
(v_1^m)_{ss}+\frac{n-2-(n+2)m}{1-m}(v_1^m)_s-\frac{2m(n-2-nm)}{(1-m)^2}v_1^m
+\frac{m\beta}{n-1}v_{1,s}+\frac{m\rho_1}{(1-m)(n-1)}v_1=0
\end{equation}
in $(-\infty,\infty)$ and $w$ satisfies
\begin{equation}\label{w1-eqn}
w_{ss}=\frac{1-2m}{1-m}\cdot\frac{w_s^2}{w}
-\frac{n-2-(n+2)m}{1-m}w_s-\frac{\beta}{n-1}ww_s
-\frac{\rho_1}{n-1}w^2+\frac{2(n-2-nm)}{1-m}w
\end{equation}
in $(-\infty,\infty)$ or equivalently
\begin{equation}\label{w1-r-eqn}
w_{rr}+\left(1+\frac{n-2-(n+2)m}{1-m}\right)\frac{w_r}{r}
-\frac{1-2m}{1-m}\cdot\frac{w_r^2}{w}
+\frac{\beta}{n-1}\frac{ww_r}{r}
+\frac{\rho_1}{n-1}\frac{w^2}{r^2}-\frac{2(n-2-nm)}{1-m}\frac{w}{r^2}=0
\end{equation}
in $(0,\infty)$. When $\rho_1=0$, \eqref{v1-eqn}, \eqref{w1-eqn},
and \eqref{w1-r-eqn} reduce
to
\begin{equation}\label{v1-eqn2}
(v_1^m)_{ss}+\frac{n-2-(n+2)m}{1-m}(v_1^m)_s
-\frac{2m(n-2-nm)}{(1-m)^2}v_1^m+\frac{m\beta}{n-1}v_{1,s}=0
\quad\mbox{ in }(-\infty,\infty),
\end{equation}
\begin{equation}\label{w1-eqn2}
w_{ss}=\frac{1-2m}{1-m}\cdot\frac{w_s^2}{w}
-\frac{n-2-(n+2)m}{1-m}w_s-\frac{\beta}{n-1}ww_s
+\frac{2(n-2-nm)}{1-m}w\quad\mbox{ in }(-\infty,\infty)
\end{equation}
and
\begin{equation}\label{w1-r-eqn2}
w_{rr}+\left(1+\frac{n-2-(n+2)m}{1-m}\right)\frac{w_r}{r}
-\frac{1-2m}{1-m}\cdot\frac{w_r^2}{w}
+\frac{\beta}{n-1}\frac{ww_r}{r}
-\frac{2(n-2-nm)}{1-m}\frac{w}{r^2}=0
\end{equation}
in $(0,\infty)$.

\begin{lem}\label{w1s-bd}
Let $\eta>0$ and let $m$, $\alpha$, $\beta$, satisfy \eqref{m-range} and
\eqref{alpha-beta-relation3}. Let $v$ be the radially symmetric solution of 
\eqref{elliptic-eqn}. Then there exist constants $C_1>0$, $C_2>0$, $C_3>0$, 
such that
\begin{equation}\label{rw'-w-upper-bd}
\frac{rw_r(r)}{w(r)}\le C_1\quad\forall r\ge 0
\end{equation}
and
\begin{equation}\label{r-w1'-lower-bd}
C_2\le rw_r(r)\le C_3\quad\forall r\ge 1.
\end{equation}
Moreover 
\begin{equation}\label{w1-limit}
w(r)\to\infty\quad\mbox{ as }r\to\infty.
\end{equation}
\end{lem}
\begin{proof}
Note that $v_1(-\infty)=v_{1,s}(-\infty)=0$ and by 
Lemma \ref{h-lower-bd} $v_{1,s}>0$ on $(-\infty,\infty)$. Let
\begin{equation}\label{b0-defn}
b_0=\frac{n-2-(n+2)m}{1-m}\quad\mbox{ and }\quad b_1
=\frac{2m(n-2-nm)}{(1-m)^2}.
\end{equation}
If $b_0\ge 0$, then by \eqref{v1-eqn2},
\begin{equation*}
(v_1^m)_{ss}-b_1v_1^m\le 0\quad\Rightarrow\quad (v_1^m)_s\le b_1v_1^m
\quad\Rightarrow\quad \frac{rw_r(r)}{w(r)}\le\frac{(1-m)b_1}{m}\quad\forall
r\ge 0
\end{equation*}
and \eqref{rw'-w-upper-bd} follows. 

If $b_0<0$, by \eqref{v1-eqn2},
\begin{equation}\label{vm'-ineqn}
(v_1^m)_{ss}+b_0(v_1^m)_s-b_1v_1^m\le 0.
\end{equation}
Let $p=(v_1^m)_s/v_1^m$. Then by \eqref{vm'-ineqn},
\begin{equation}\label{p-eqn}
p_s=\frac{(v_1^m)_{ss}}{v_1^m}-\frac{(v_1^m)_s^2}{v_1^{2m}}\le |b_0|p+b_1-p^2
=-(p-(|b_0|/2))^2+b_1+(b_0^2/4).
\end{equation}
Let 
\begin{equation*}
b_2=\max\left(\frac{3m}{1-m},\sqrt{b_1+b_0^2}+|b_0|\right).
\end{equation*}
We claim that 
\begin{equation}\label{p-claim}
p(s)\le b_2\quad\forall s\in\R.
\end{equation}
Suppose \eqref{p-claim} does not hold. Then there exists $s_0\in\R$ such 
that $p(s_0)>b_2$. Since
\begin{equation}\label{p-eqn2}
p=m\frac{v_{1,s}}{v_1}=\frac{m}{1-m}\frac{w_s}{w}=\frac{m}{1-m}\frac{rw_r}{w}
=\frac{2m}{1-m}\left(1+\frac{1-m}{2}\cdot\frac{rv_r(r)}{v(r)}\right),
\end{equation}
$p(s=-\infty)=2m/(1-m)$. Let $s_1=\inf\{s'<s_0:p(s)>b_2\quad\forall 
s'\le s\le s_0\}$. Then $-\infty<s_1<s_0$, $p(s)>b_2$ for any
$s\in (s_1,s_0)$, and $p(s_1)=b_2$. By \eqref{p-eqn}, $p_s(s)<0$ 
for any $s\in (s_1,s_0)$. Hence $p(s_0)\le p(s_1)=b_2$. Thus contradiction
arises and \eqref{p-claim} follows. Then by \eqref{p-claim} and 
\eqref{p-eqn2}, \eqref{rw'-w-upper-bd} holds with $C_1=b_2/m$.

Let
$$
a_1=\frac{2(n-2-nm)}{1-m},\quad a_2=\frac{\beta}{(n-1)a_1},\quad
\mbox{ and }a_3=a_1^{-1}\max(|b_0|,|1-2m|/(|1-m|w(1))).
$$
Since $w_s>0$ for any $s\in\R$, $w(s)\ge w(1)$ for any $s\ge 1$. Then by 
\eqref{w1-eqn2},
\begin{equation}\label{wss-lower-bd}
w_{ss}\ge a_1((1-a_2w_s)w-a_3(w_s+w_s^2))\quad\forall s\ge 1.
\end{equation}
Suppose $w_s\le C_2':=\min (1,(2a_2)^{-1},w(1)/(8a_3))$ for all $s\ge 1$. 
Then by \eqref{wss-lower-bd}, 
\begin{equation}\label{wss-lower-bd2}
w_{ss}\ge a_1w(1)/4>0\quad\forall s\ge 1.
\end{equation}
Hence 
$w_s\to\infty$ as $s\to\infty$ and contradiction arises. Thus there exists
$s_1>1$ such that $w_s(s_1)>C_2'$. Suppose there exists $s_2>s_1$ such that
$w_s(s_2)<C_2'$. Let $s_3=\inf\{s'<s_2:w_s(s)<C_2'\quad\forall s'\le s\le s_2\}$.
Then $s_1<s_3<s_2$ and $w_s(s_3)=C_2'$. Then by the above argument
\eqref{wss-lower-bd2} holds in $(s_3,s_2)$. Hence $w_s(s_2)>w_s(s_3)=C_2'$ 
and contradiction arises. Thus $w_s(s)\ge C_2'$ for any $s\ge s_1$. Since
$w_s(s)>0$ for all $s\in\R$, the left hand side of \eqref{r-w1'-lower-bd} 
holds with $C_2=\min (C_2',\min_{[0,s_1]}w_s(s))>0$ and \eqref{w1-limit} holds.

Let
$$
a_4=\frac{\beta}{3(n-1)}\quad\mbox{ and }\quad
a_5=\frac{\beta(1-m)}{3(n-1)}.
$$
By \eqref{w1-eqn2} and \eqref{rw'-w-upper-bd},
\begin{equation}\label{wss-upper-ineqn}
w_{ss}\le\left\{\begin{aligned}
&(|b_0|-a_4w)w_s+a_1w(1-(a_2/3)w_s)+(1-m)^{-1}w_s
[(1-2m)C_1-a_5w]\quad\mbox{ if }0<m<1/2\\
&(|b_0|-a_4w)w_s+a_1w(1-(a_2/3)w_s)\qquad\qquad\qquad\qquad\qquad\quad
\mbox{ if }1/2\le m<(n-2)/n.
\end{aligned}\right.
\end{equation}
By \eqref{w1-limit} there exists a constant $s_0>0$ such that
\begin{equation}\label{w-lower-bd}
w>\max((1-2m)C_1/a_5,|b_0|/a_4)\quad\forall s\ge s_0.
\end{equation}
By \eqref{wss-upper-ineqn} and \eqref{w-lower-bd},
\begin{equation}\label{wss-upper-ineqn2}
w_{ss}\le a_1w(1-(a_2/3)w_s)\quad\forall s\ge s_0.
\end{equation}
We claim that there exists a constant $s_1'>s_0$ such that
\begin{equation}\label{ws-lower-bd}
w_s\le C_3':=\max(5/a_2,2w_s(s_0))\quad\forall s\ge s_1'.
\end{equation}
Suppose \eqref{ws-lower-bd} does not hold. Then there exists a constant
$s_2'>s_0$ such that 
$$
w_s(s_2')>C_3'.
$$
Let $s_3'=\inf\{s_0\le t_0<s_2':w_s(s)>C_3'\quad\forall t_0\le s\le s_2'\}$. 
Then $s_0<s_3'<s_2'$, $w_s>C_3'$ for any $s_3'<s<s_2'$, and $w_s(s_3')=C_3'$. 
Then by \eqref{wss-upper-ineqn2} $w_{ss}<0$ in $(s_3',s_2')$. Hence 
$w_s(s_2')\le w_s(s_3')=C_3'$ and contradiction arises. Thus no such
constant $s_2'$ exists and there exists a constant $s_1'>s_0$ such that
\eqref{ws-lower-bd} holds. Then the right hand side of 
\eqref{r-w1'-lower-bd} holds with $C_3=\max (C_3',\max_{[0,s_1']}w_s(s))>0$.
\end{proof}

\begin{thm}\label{decay-rate-thm}
Let $\eta>0$ and let $m$, $\alpha$, $\beta$, satisfy \eqref{m-range} and
\eqref{alpha-beta-relation3}. Let $v$ be the radially 
symmetric solution of \eqref{elliptic-eqn}. Then \eqref{decay-rate2} holds.
\end{thm}
\begin{proof}
Let $q(r)=rw_r(r)$, 
$$
a_0=\frac{2(n-2-nm)(n-1)}{(1-m)\beta},\quad q_1=q-a_0,
$$
and let $b_0$ be given by \eqref{b0-defn}.
By \eqref{w1-r-eqn2},
\begin{align}
&(r^{b_0}q(r)w(r)^{\frac{2m-1}{1-m}})'=\frac{\beta}{n-1}
\cdot\frac{w^{\frac{m}{1-m}}}{r^{1-b_0}}(a_0-q(r))\quad\forall r>0\label{q-eqn}\\
\Rightarrow\quad&q_r+\frac{b_0}{r}q+\frac{\beta}{n-1}\frac{w}{r}
(q-a_0)=\frac{1-2m}{1-m}\cdot\frac{q^2}{rw}\quad\forall r>0\notag\\
\Rightarrow\quad&q_{1,r}+\frac{b_0}{r}q_1+\frac{\beta}{n-1}\frac{w}{r}q_1
=\frac{1-2m}{1-m}\cdot\frac{q^2}{rw}-\frac{b_0a_0}{r}\quad\forall r>0.
\label{q1-eqn}
\end{align}
Since
\begin{align*}
&r^{b_0}q(r)w(r)^{\frac{2m-1}{1-m}}=r^{b_0}\cdot(r^2v(r)^{1-m})^{\frac{2m-1}{1-m}}
\cdot 2r^2v(r)^{-m}h(r)=2r^{\frac{n-2-nm}{1-m}}v(r)^m\left(1+\frac{1-m}{2}
\cdot\frac{rv_r(r)}{v(r)}\right)\\
\Rightarrow\quad&\lim_{r\to 0}r^{b_0}q(r)w(r)^{\frac{2m-1}{1-m}}=0,
\end{align*}
integrating \eqref{q-eqn} over $(0,r)$,
\begin{align}
&r^{b_0}q(r)w(r)^{\frac{2m-1}{1-m}}=\frac{\beta}{n-1}
\int_0^r\rho^{b_0-1}w(\rho)^{\frac{m}{1-m}}(a_0-q(\rho))\,d\rho\label{q-eqn2}\\
\Rightarrow\quad&q(r)=\frac{\beta}{n-1}\cdot\frac{\int_0^r\rho^{b_0-1}
w(\rho)^{\frac{m}{1-m}}(a_0-q(\rho))\,d\rho}{r^{b_0}w(r)^{\frac{2m-1}{1-m}}}
\label{q-eqn3}
\end{align}
Let $\{r_i\}_{i=1}^{\infty}$ be a sequence of positive numbers such 
that $r_i\to\infty$ as $i\to\infty$. By Lemma \ref{w1s-bd} there exist
constants $C_1>0$, $C_2>0$, $C_3>0$, such that \eqref{rw'-w-upper-bd} and 
\eqref{r-w1'-lower-bd} holds. Then by \eqref{r-w1'-lower-bd} the sequence 
$\{r_i\}_{i=1}^{\infty}$ has a subsequence which we may assume without loss 
of generality to be the sequence itself such that 
$q(r_i)\to q_{\infty}$ as $i\to\infty$ for some constant $q_{\infty}$
satisfying 
\begin{equation}\label{q-limit-lower-bd}
C_2\le q_{\infty}\le C_3.
\end{equation}
Suppose $q_{\infty}\ne a_0$. Let 
$$
f_1(r)=\mbox{exp}\,\left(\frac{\beta}{n-1}\int_0^r\rho^{-1}w(\rho)
\,d\rho\right).
$$
We now divide the proof into three cases.

\noindent $\underline{\text{\bf Case 1}}$: $(n-2)/(n+2)<m=1/2<(n-2)/n$.

\noindent 
Then $b_0<0$. Since $m=1/2$, by \eqref{q1-eqn},
\begin{equation}\label{q-bd}
(r^{b_0}f_1(r)q_1(r))'\le 0\quad\Rightarrow\quad q_1(r)\le 0\quad\Rightarrow
\quad
0\le q(r)\le a_0\quad\forall r>0.
\end{equation}
Hence $q_{\infty}<a_0$. Then by \eqref{q-eqn3} and \eqref{q-bd},
\begin{equation*}
q(r_i)=\frac{\beta}{n-1}r_i^{|b_0|}\int_0^{r_i}\rho^{b_0-1}
w(\rho)^{\frac{m}{1-m}}(a_0-q(\rho))\,d\rho\to\infty\quad\mbox{ as }i\to\infty
\end{equation*}
and contradiction arises. Hence $q_{\infty}=a_0$.

\noindent $\underline{\text{\bf Case 2}}$: $(n-2)/(n+2)<m<(n-2)/n$ and 
$m\ne 1/2$.

\noindent
Since $b_0<0$, by Lemma \ref{w1s-bd}, 
\eqref{q-eqn2}, \eqref{q-eqn3}, \eqref{q-limit-lower-bd}, and 
the l'Hosiptal rule,
\begin{align*}
q_{\infty}=&|\lim_{i\to\infty}q(r_i)|
=\frac{\beta}{n-1}\left|\lim_{i\to\infty}\frac{r_i^{|b_0|}\int_0^{r_i}\rho^{b_0-1}
w(\rho)^{\frac{m}{1-m}}(a_0-q(\rho))\,d\rho}{w(r_i)^{\frac{2m-1}{1-m}}}\right|\\
=&\frac{\beta}{n-1}\left|\lim_{i\to\infty}
\frac{|b_0|r_i^{|b_0|-1}\int_0^{r_i}\rho^{b_0-1}w(\rho)^{\frac{m}{1-m}}
(a_0-q(\rho))\,d\rho+
r_i^{-1}w(r_i)^{\frac{m}{1-m}}(a_0-q(r_i))}{\frac{2m-1}{1-m}
w(r_i)^{\frac{3m-2}{1-m}}w_r(r_i)}\right|\\
=&\frac{\beta(1-m)}{(n-1)|2m-1|}q_{\infty}^{-1}
\left|\lim_{i\to\infty}[|b_0|(n-1)\beta^{-1}q_{\infty}
w(r_i)+w(r_i)^2(a_0-q_{\infty})]\right|\\
=&\infty.
\end{align*}
Hence contraction arises. Thus $q_{\infty}=a_0$. 

\noindent $\underline{\text{\bf Case 3}}$: $0<m\le (n-2)/(n+2)$.

\noindent
Then $b_0\ge 0$. By \eqref{q1-eqn},
\begin{align}\label{q1-integral-eqn}
&r^{b_0}f_1(r)q_1(r)=f_1(1)q_1(1)-a_0b_0\int_1^r\rho^{b_0-1}f_1(\rho)\,d\rho
+\frac{1-2m}{1-m}\int_1^r\frac{\rho^{b_0-1}q(\rho)^2f_1(\rho)}{w(\rho)}
\,d\rho\quad\forall r\ge 1\notag\\
\Rightarrow\quad&q_1(r)=\frac{f_1(1)q_1(1)
-a_0b_0\int_1^r\rho^{b_0-1}f_1(\rho)\,d\rho+\frac{1-2m}{1-m}
\int_1^r\frac{\rho^{b_0-1}q(\rho)^2f_1(\rho)}{w_1(\rho)}\,d\rho}{r^{b_0}f_1(r)}
\quad\forall r\ge 1.
\end{align}
By \eqref{rw'-w-upper-bd}, \eqref{w1-limit}, and the l'Hosiptal rule,
\begin{equation*}
\liminf_{r\to\infty}\frac{f_1(r)}{w(r)}=\frac{\beta}{n-1}\liminf_{r\to\infty}
\frac{r^{-1}w(r)f_1(r)}{w'(r)}
\ge\frac{\beta}{(n-1)C_1}\liminf_{r\to\infty}f_1(r)=\infty.
\end{equation*}
Thus there exists a constant $R_1>1$ such that
\begin{equation}\label{f-w-lower-bd}
\frac{f_1(r)}{w(r)}\ge 1\quad\forall r\ge R_1.
\end{equation}
By \eqref{r-w1'-lower-bd} and \eqref{f-w-lower-bd},
\begin{equation}\label{f2-integral-infty}
\int_1^r\frac{\rho^{b_0-1}q(\rho)^2f_1(\rho)}{w(\rho)}\,d\rho
\ge C\int_{R_1}^r\rho^{b_0-1}\,d\rho\to\infty\quad\mbox{ as }r\to\infty.
\end{equation}
On the other hand
\begin{equation}\label{limit1}
\lim_{r\to\infty}\frac{\int_1^r\rho^{b_0-1}f_1(\rho)\,d\rho}{r^{b_0}f_1(r)}
=\lim_{r\to\infty}
\frac{r^{b_0-1}f_1(r)}{b_0r^{b_0-1}f_1(r)+\beta (n-1)^{-1}r^{b_0-1}w(r)f_1(r)}
=\lim_{r\to\infty}\frac{1}{b_0+\beta (n-1)^{-1}w(r)}=0.
\end{equation}
By \eqref{q-limit-lower-bd}, \eqref{q1-integral-eqn}, 
\eqref{f2-integral-infty}, \eqref{limit1} and the l'Hosiptal rule,
\begin{align*}
\lim_{i\to\infty}q_1(r_i)=&\lim_{i\to\infty}\frac{f_1(1)q_1(1)
-a_0b_0\int_1^r\rho^{b_0-1}f_1(\rho)\,d\rho+\frac{1-2m}{1-m}
\int_1^{r_i}\frac{\rho^{b_0-1}q(\rho)^2f_1(\rho)}{w(\rho)}\,d\rho}
{r_i^{b_0}f_1(r_i)}\\
=&\frac{1-2m}{1-m}\lim_{i\to\infty}\frac{r_i^{b_0-1}q(r_i)^2w(r_i)^{-1}f_1(r_i)}
{b_0r_i^{b_0-1}f_1(r_i)+\beta(n-1)^{-1}r_i^{b_0-1}w(r_i)f_1(r_i)}\\
=&\frac{1-2m}{1-m}\lim_{i\to\infty}\frac{q(r_i)^2w(r_i)^{-2}}
{b_0w(r_i)^{-1}+\beta(n-1)^{-1}}\\
=&0
\end{align*}
Hence $q_{\infty}=a_0$. 

By case 1, case 2, and case 3, $q(r_i)\to a_0$ as $i\to\infty$. Since 
the sequence $\{r_i\}_{i=1}^{\infty}$ is arbitrary, $q(r)\to a_0$ as $r\to\infty$.
\end{proof}

\begin{cor}
The metric $g_{ij}=v^{\frac{4}{n+2}}dx^2$, $n\ge 3$, of a locally conformally 
flat gradient steady Yamabe soliton where $v$ satisfies \eqref{elliptic-eqn}
has the exact decay rate \eqref{decay-rate}. 
\end{cor}

By Theorem \ref{m-to-0-thm}, Theorem \ref{decay-rate-thm}, and the result
of \cite{Hs} we have the following result.

\begin{cor}
Let $\beta>0$, $\eta>0$, and $n\ge 3$. For any $0<m<(n-2)/n$, let 
$\alpha_m=2\beta/(1-m)$ and let $v^{(m)}$ be the radially solution of 
\eqref{elliptic-eqn} with $\alpha=\alpha_m$. Then
$$
\lim_{|x|\to\infty}\lim_{m\to 0}\frac{|x|^2v^{(m)}(x)^{1-m}}{\log |x|}
=\lim_{m\to 0}\lim_{|x|\to\infty}\frac{|x|^2v^{(m)}(x)^{1-m}}{\log |x|}
=\frac{2(n-1)(n-2)}{\beta}.
$$ 
\end{cor}

\section{Decay rate for $\frac{2\beta}{1-m}>\max(\alpha,0)$}
\setcounter{equation}{0}
\setcounter{thm}{0}

In this section we will use a modification of the technique of \cite{Hs} 
to prove the decay rate \eqref{decay-rate3} of the radially symmetric solution 
of \eqref{elliptic-eqn} when \eqref{m-range} and \eqref{alpha-beta-relation5} 
hold. 

\begin{thm}
Let $\eta>0$ and let $m$, $n$, $\alpha$, $\beta$, satisfies \eqref{m-range} 
and \eqref{alpha-beta-relation5}. Let $v$ be the solution of 
\eqref{elliptic-eqn}.
Then \eqref{decay-rate3} holds for some constant $A>0$.
\end{thm}
\begin{proof}
Let $q(r)=r^{\alpha/\beta}v(r)$. Then by Lemma \ref{h1-lower-bd},
\begin{equation}\label{q'-lower-bd3.1}
q'(r)=(\alpha/\beta)r^{\frac{\alpha}{\beta}-1}(v(r)+krv'(r))>0\quad\forall r>0.
\end{equation}
By direct computation,
\begin{equation}\label{q-eqn3.1}
\left(\frac{q'}{q}\right)'+\frac{n-1-(2m\alpha/\beta)}{r}\cdot\frac{q'}{q}
+m\left(\frac{q'}{q}\right)^2
+\frac{\beta r^{1-\frac{\alpha}{\beta}(1-m)}q'}{(n-1)q^m}
=\frac{\alpha}{\beta}\cdot\frac{n-2-(m/k)}{r^2}.
\end{equation}
Let
$$
f_2(r)=\mbox{exp}\,\left(\frac{\beta}{n-1}\int_1^r
\rho^{1-\frac{\alpha}{\beta}(1-m)}q(\rho)^{1-m}\,d\rho\right).
$$
Then $f_2'(r)=(n-1)^{-1}\beta r^{1-\frac{\alpha}{\beta}(1-m)}q(r)^{1-m}f_2(r)$ and
\begin{equation}\label{f2-limit}
f_2(r)\ge \mbox{exp}\,\left(\frac{\beta q(1)^{1-m}}{n-1}
\int_1^r\rho^{1-\frac{\alpha}{\beta}(1-m)}\,d\rho\right)
=\mbox{exp}\,\left(c_0(r^{2-\frac{\alpha}{\beta}(1-m)}-1)\right)\to\infty
\quad\mbox{ as }r\to\infty.
\end{equation}
where 
$$
c_0=\frac{\beta q(1)^{1-m}}{(n-1)(2-\frac{\alpha}{\beta}(1-m))}.
$$
Let $c_1=q(1)^{m-1}q'(1)f_2(1)$ and $c_2=(\alpha/\beta)(n-2-(m/k))$. Multiplying 
\eqref{q-eqn3.1} by $r^{n-1-(2m\alpha/\beta)}q(r)^mf_2(r)$ and integrating over $(1,r)$,
\begin{equation}\label{q'-integral-eqn3.1}
r^{n-1-(2m\alpha/\beta)}q(r)^mf_2(r)\cdot\frac{q'(r)}{q(r)}
=c_1+c_2\int_1^r\rho^{n-3-(2m\alpha/\beta)}q(\rho)^mf_2(\rho)\,d\rho
\quad\forall r>1.
\end{equation}
By \eqref{f2-limit}, \eqref{q'-integral-eqn3.1}, and the l'Hosiptal rule,
\begin{align}\label{q'/q-limit}
\limsup_{r\to\infty}r^p\frac{q'(r)}{q(r)}
\le&\limsup_{r\to\infty}\frac{c_1+c_2\int_1^r\rho^{n-3-(m\alpha/\beta)}
q(\rho)^mf_2(\rho)\,d\rho}{r^{n-p-1-(2m\alpha/\beta)}q(r)^mf_2(r)}\notag\\
\le&c_2\limsup_{r\to\infty}\frac{r^{n-3-(2m\alpha/\beta)}q(r)^mf_2(r)}{F(r)}
\quad\forall p>0
\end{align}
where 
\begin{align}\label{F-lower-bd}
F(r)=&(n-p-1-(2m\alpha/\beta))r^{n-p-2-(2m\alpha/\beta)}q(r)^mf_2(r)
+mr^{n-p-1-(2m\alpha/\beta)}q(r)^{m-1}q'(r)f_2(r)\notag\\
&\qquad +r^{n-p-1-(2m\alpha/\beta)}q(r)^mf_2'(r)\notag\\
\ge&(n-p-1-(2m\alpha/\beta))r^{n-p-2-(2m\alpha/\beta)}q(r)^mf_2(r)
+r^{n-p-1-(2m\alpha/\beta)}q(r)^mf_2'(r)\notag\\
=&(n-p-1-(2m\alpha/\beta))r^{n-p-2-(2m\alpha/\beta)}q(r)^mf_2(r)
+(n-1)^{-1}\beta r^{n-p-(\alpha/\beta)(1+m)}q(r)f_2(r)
\end{align}
By \eqref{q'-lower-bd3.1}, \eqref{q'/q-limit} and \eqref{F-lower-bd},
\begin{align*}
0\le&\limsup_{r\to\infty}r^p\frac{q'(r)}{q(r)}\\
\le&c_2\limsup_{r\to\infty}\frac{r^{n-3-(2m\alpha/\beta)}q(r)^mf_2(r)}
{(n-p-1-\frac{2m\alpha}{\beta})r^{n-p-2-(2m\alpha/\beta)}q(r)^mf_2(r)
+\beta(n-1)^{-1}r^{n-p-(\alpha/\beta)(1+m)}q(r)f_2(r)}\\
\le&c_2\limsup_{r\to\infty}\frac{1}{(n-p-1-\frac{2m\alpha}{\beta})r^{1-p}
+\beta(n-1)^{-1}r^{3-p-(\alpha/\beta)(1-m)}q(r)^{1-m}}\\
=&0\qquad\qquad\forall 1<p<3-(\alpha/\beta)(1-m).
\end{align*}
Hence
\begin{equation}\label{r-q'-q-limit}
\lim_{r\to\infty}r^p\frac{q'(r)}{q(r)}=0\quad\forall 1<p<3-(\alpha/\beta)(1-m).
\end{equation}
Let $p_0=2-(\alpha/2\beta)(1-m)$. By \eqref{alpha-beta-relation5}, 
$1<p_0<3-(\alpha/\beta)(1-m)$. By \eqref{r-q'-q-limit},
\begin{equation}
|\log q(r)-\log q(1)|\le\int_1^r|(\log q)'(\rho)|\,d\rho\le C\int_1^r\rho^{-p_0}
\,d\rho\le C_3\quad\forall r\ge 1.
\end{equation}
for some constant $C_3>0$. Hence
\begin{equation}\label{q-upper-lower-bd3.1}
e^{-C_3}q(1)\le q(r)\le e^{C_3}q(1)\quad\forall r\ge 1.
\end{equation}
By \eqref{q'-lower-bd3.1} and \eqref{q-upper-lower-bd3.1}, $q(r)$ increases 
to some constant $A\in\R$ as $r\to\infty$ and the theorem follows.
\end{proof}

\end{document}